\documentclass[reqno]{amsart}
\usepackage{amsmath, amsthm, amssymb, xypic, enumitem, iitem}
\usepackage{hyperref}
\usepackage{verbatim} 
\usepackage{tikz}
\usepackage{tikz-cd}
\usepackage{color}
\usepackage{cite}
\usepackage{setspace}
\usepackage{physics}
\usepackage{adjustbox}
\usepackage{todonotes}
\usetikzlibrary{nfold}

\calclayout

\setenumerate{label=\textnormal{(\arabic*)}} 

\newcommand{\A}{\mathcal{A}}
\newcommand{\C}{\mathcal{C}}
\newcommand{\D}{\mathcal{D}}
\newcommand{\X}{\mathcal{X}}
\newcommand{\Y}{\mathcal{Y}}
\newcommand{\Z}{\mathcal{Z}}
\renewcommand{\A}{\mathcal{A}}

\newcommand{\W}{\mathcal{W}}
\newcommand{\T}{\mathcal{T}}

\newcommand{\ob}{\text{Ob}}
\newcommand{\mor}{\text{Mor}}
\newcommand{\smc}{\mathbf{SymMonCat}_\mathrm{s}}
\newcommand{\psh}{\X \star_\A \Y}
\newcommand{\thr}{\T} 
\newcommand{\ntns}{\mathbin{\boxtimes}}
\newcommand{\ncmp}{%
  \mathbin{\text{\fboxsep=-.2pt\fbox{\rule{0pt}{1ex}\rule{1ex}{0pt}}}}%
} 
\newcommand{\obs}{\mathcal{O}} 
\newcommand{\obt}{\mathcal{O}'} 
\newcommand{\pmaf}{\mathcal{F}} 
\newcommand{\pmas}{\pmaf'} 
\newcommand{\pma}{\mathcal{M}} 
\newcommand{\cta}{\T} 
\DeclareMathOperator{\zi}{ZI}
\DeclareMathOperator{\tp}{tp} 
\DeclareMathOperator{\id}{id}

\numberwithin{equation}{section}

\theoremstyle{plain}
\newtheorem{theorem}[equation]{Theorem}
\newtheorem{corollary}[equation]{Corollary}

\newtheorem{proposition}[equation]{Proposition}

\theoremstyle{definition}
\newtheorem{definition}[equation]{Definition}

\newlist{romanlist}{enumerate}{1}
\setlist[romanlist, 1]{
  label=\roman{romanlisti}.,
  ref=\arabic{section}.\arabic{equation}.\roman{romanlisti},
}

\usepackage[foot]{amsaddr}

\title{On pushouts of strict symmetric monoidal categories}

\author{Lucas Anderson$^1$}
\address{$^1$Oregon State University}
\author{Ariel E. Rosenfield$^2$}
\address{$^2$University of Maryland, College Park}
\author{Eric Yu$^2$}
\email{anderl23@oregonstate.edu, ariari@umd.edu, yue8118@terpmail.umd.edu}

\begin{document}
	\begin{spacing}{1.1}

    \begin{abstract} Toward modeling computational concurrency using the language of central idempotents in monoidal categories, we give an explicit construction of the pushout of strict symmetric monoidal categories. We show that any braided strong monoidal functor strongly preserves central idempotents, and as a corollary, that the lattice of central idempotents of a symmetric monoidal category is isomorphic as a meet-semilattice to that of its strictification.
\end{abstract}

    \maketitle

\date{\today}

\setcounter{tocdepth}{1} 
\makeatletter
\def\l@subsection{\@tocline{2}{0pt}{2.5pc}{5pc}{}}
\def\l@subsubsection{\@tocline{2}{0pt}{5pc}{7.5pc}{}} 
\makeatother
\tableofcontents

\section{Introduction}

Monoidal categories are a natural setting for describing phenomena in quantum computation: They admit inbuilt notions of both parallel and sequential composition of operations on multipartite systems, allowing us to keep track of which operations are performed on such a system, and when \cite{cqt}. Even further, any monoidal category comes equipped with a notion of ``where," or on which part of a system, a given operation is performed; namely, a monoidal category $\C$ satisfying a mild condition called \textit{firmness} (defined in \cite{tensor-topology}) is equipped with a meet-semilattice $\zi(\C)$ of so-called central idempotents, and any morphism $f$ of $\C$ can be said to restrict to a unique smallest element of $\zi(\C)$, called its support, and denoted $\text{supp}(f)$. As in \cite[\S 5]{localisable-monads}, the lattice of central idempotents in a monoidal category $\C$ may be interpreted, for example, as a topological space of locations in a computer memory, or of the physical locations of a network of interacting agents. 

Quantum computation naturally exhibits concurrent behavior---for example, performing a local operation on one of an entangled pair of qubits affects the other member of the pair. It is therefore of interest in quantum computing applications to find structures in monoidal categories which model concurrency. Specifically, in the present work, we would like to find a structure which allows us to ensure that concurrent operations can exchange computational resources. This is meant in the sense that, for example, quantum teleportation is an exchange of a quantum state between communicating parties. In \cite{space-in-moncats}, the authors show that quantum teleportation, viewed as an operation on a bipartite system, can only occur if the two communicating qubits have non-disjoint support. Inspired by this example, we will say that two morphisms $f,g$ in a monoidal category $\C$ can \textit{exchange resources} where $\text{supp}(f)$ and $\text{supp}(g)$ overlap.

It is easy to see, since $\zi(\C)$ is a meet-semilattice when $\C$ is a firm monoidal category, that any two processes occurring in $\C$ can exchange resources. The difficulty arises in modeling resource exchange between a process $f$ occurring in a braided monoidal category $\X$ and a process $g$ occurring in a separate braided monoidal category $\Y$---such a situation might occur, for example, if we view $\X$ and $\Y$ as modeling the memories of two separate computers. In this case, we would like to construct a third monoidal category $\mathcal{W}$ containing isomorphic copies of $\X$ and $\Y$, such that the lattice $\zi(\mathcal{W})$ contains the central idempotent lattices of the two original categories, together with meets of the elements thereof. Constructing a $\mathcal{W}$ containing isomorphic copies of $\X$ and $\Y$ is thus part of the way toward allowing us to enforce the capacity for resource exchange. 

To this end, given strict symmetric monoidal categories $\X$ and $\Y$ both admitting strong symmetric monoidal functors from a symmetric monoidal category $\A$, we construct the pushout $\psh$ of $\X$ and $\Y$ over $\A$ in \textsection \ref{sec:pushout}. For completeness, we mention that similar constructions were outlined in \cite{macdonald-scull} for the case of general categories with monomorphisms between them; as well as in \cite{quantum-pi} for symmetric monoidal categories which both have the same objects; and in \cite{props} for PROPs. In \textsection \ref{sec:central-idempotents}, we show that $\zi(\psh)$ contains the central idempotent lattices of $\X$ and $\Y$; and that if $\C$ is merely symmetric monoidal, $\zi(\C)$ is lattice-isomorphic to the central idempotent lattice of the strictification of $\C$.

We defer to future work the question of whether $\X$ and $\Y$ being stiff (in the sense of \cite{sheaf-rep}) or firm, or having universal joins of central idempotents, implies that the pushout retains any of these properties. 

\subsection{Acknowledgements} This project was initiated as part of the ACT Adjoint School 2023. The first author thanks Chris Heunen, Nesta van der Schaaf, and Carmen Constantin for their guidance and feedback during and beyond the term of the school, and also thanks fellow group members Clem\'ence Chanavat, Ariadne Si Suo, and Luke Morris. The first and second authors acknowledge the financial support of the NSF MathQuantum Research Training Group at University of Maryland, College Park.

\section{Preliminaries}

For \textsection \ref{sec:pushout}, we assume familiarity with Lawvere theories and their algebras, for which a comprehensive reference can be found in \cite{adamek-rosicky}. In preparation for the discussion in \textsection \ref{sec:central-idempotents}, we recall the basic theory of central idempotents in braided monoidal categories. 

\begin{definition} Let $\C$ be a monoidal category and $U \in \ob(\C)$. A \textbf{half-braiding} on $U$ is a natural isomorphism $\sigma: U \otimes - \to - \otimes U$ such that for any $A, B \in \ob(\C)$, the following diagram commutes:
    \[
      \begin{tikzcd}
        {U \otimes (A \otimes B)} && {(U \otimes A) \otimes B} && {(A \otimes U) \otimes B} \\
        \\
        {(A \otimes B) \otimes U} && {A \otimes (B \otimes U)} && {A \otimes (U \otimes B)}
        \arrow["{{\alpha^{-1}_{U A B}}}", from=1-1, to=1-3]
        \arrow["{{\sigma_{A \otimes B}}}", from=1-1, to=3-1]
        \arrow["{{\sigma_A \otimes B}}", from=1-3, to=1-5]
        \arrow["{{\alpha_{A U B}}}", from=1-5, to=3-5]
        \arrow["{{\alpha^{-1}_{A B U}}}"', from=3-3, to=3-1]
        \arrow["{{A \otimes \sigma_B}}"', from=3-5, to=3-3]
      \end{tikzcd}
    \]
\end{definition}

\begin{definition}[Drinfel'd center]
    \cite[7.13]{Etingof_Gelaki_Nikshych_Ostrik_2015} Let $\C$ be a monoidal category. Then the \textbf{center} of $\C$ is the category $\Z(\C)$ whose:
    \begin{itemize}
        \item Objects are pairs $(U, \sigma)$ where $U \in \ob(\C)$ and $\sigma$ is a half-braiding on $U$,
        \item Morphisms $(U, \sigma) \to (V, \tau)$ are morphisms $f: U \to V$ in $\C$ that \textbf{respect half-braidings}, meaning that the following diagram commutes:
        \[
          \begin{tikzcd}
            {U \otimes A} && {V \otimes A} \\
            \\
            {A \otimes U} && {V \otimes U}
            \arrow["{f \otimes A}"{description}, from=1-1, to=1-3]
            \arrow["{\sigma_A}", from=1-1, to=3-1]
            \arrow["{\tau_A}", from=1-3, to=3-3]
            \arrow["{A \otimes f}"{description}, from=3-1, to=3-3]
          \end{tikzcd}
        \]
        \item Composition and identity morphisms are as in $\C$.
    \end{itemize}
\end{definition}

Recall that there is a canonical projection functor $\Z(\C) \to \C$ from the center of $\C$ which forgets half-braidings. For convenience, we also recall the following definition, given in \cite{sheaf-rep}.

\begin{definition} Let $\C$ be a monoidal category. Then a \textbf{central idempotent} in $\C$ is a morphism $u: (U, \sigma) \to (I, \rho^{-1} \circ \lambda)$ such that the following diagram commutes and the morphism $U \otimes U \to U$ shown is invertible:
    \[
      \begin{tikzcd}
        {U \otimes U} && {U \otimes I} \\
        \\
        {I \otimes U} && U
        \arrow["{U \otimes u}"{description}, from=1-1, to=1-3]
        \arrow["{u \otimes U}", from=1-1, to=3-1]
        \arrow["{\rho_U}", from=1-3, to=3-3]
        \arrow["{\lambda_U}"{description}, from=3-1, to=3-3]
      \end{tikzcd}
    \]
    Two central idempotents are identified if they are isomorphic in $\Z(\C)/I$, and we write $u \le v$ for central idempotents $u: (U, \sigma) \to I$ and $v: (V, \tau) \to I$ if there exists a morphism $u \to v$ in $\Z(\C)/I$. Such a morphism is necessarily unique.
\end{definition}

We will want to examine how monoidal functors interact with central idempotents in the categories of interest, for which we need the following notion.

\begin{definition}
    Let $\C$ and $\D$ be monoidal categories, $F: \C \to \D$ be a lax monoidal functor, $U \in \ob(\C)$, and $\sigma$ be a half-braiding on $U$. Then $F$ \textbf{preserves the half-braiding} $\sigma$ if there exists a half-braiding $\tau$ on $F(U)$ such that the following diagram commutes for all $A \in \ob(\C)$:
    \[
      \begin{tikzcd}
        {F(U) \otimes F(A)} && {F(U \otimes A)} \\
        \\
        {F(A) \otimes F(U)} && {F(A \otimes U)}
        \arrow["{\mu^F_{U A}}", from=1-1, to=1-3]
        \arrow["{\tau_{F(A)}}", from=1-1, to=3-1]
        \arrow["{F(\sigma_A)}", from=1-3, to=3-3]
        \arrow["{\mu^F_{A U}}", from=3-1, to=3-3]
      \end{tikzcd}
    \]
\end{definition}

\section{Pushout of strict symmetric monoidal categories} \label{sec:pushout}

Below, suppose given a diagram \begin{equation} \label{eq:diagram-for-pushout} \begin{tikzcd}
	\mathcal{X} & \mathcal{A} & \mathcal{Y}
	\arrow["G"', from=1-2, to=1-1]
	\arrow["F", from=1-2, to=1-3]
\end{tikzcd} 
\end{equation}
of strict symmetric monoidal categories and strict symmetric monoidal functors. Below, we let $\C$ denote any of the categories $\X$,$\Y$, or $\A$, and denote the monoidal unit, associator, left unitor, right unitor, and braiding on $\C$ by $I_\C, \alpha^\C$, $\lambda^\C$, $\rho^\C$, and $\sigma^\C$, respectively. We will construct the pushout $\psh$ of this diagram in several stages.

\subsection{Objects}

Because we are considering only strict monoidal categories, the sets of objects form genuine monoids under the monoidal product. Thus, we can construct $\ob(\psh)$ as a monoid, which we will see is the pushout of $\ob(\X) \xleftarrow{G} \ob(\A) \xrightarrow{F} \ob(\Y)$ in the category of monoids, as described in \cite{bourbaki}. For simplicity, we will start with the pushout in the category of semigroups:

\begin{definition}
  Let $\obs$ be the quotient of the free semigroup with operation $\ntns$ on $\ob(\X) \sqcup \ob(\Y)$ under the congruence relation generated by the identifications: 
  \begin{romanlist}
    \item $a \ntns b \sim a \otimes b$ for any $a , b \in \ob(\C)$, $\C \in \{\X , \Y\}$
    \item $F(a) \sim G(a)$ for any $a \in \ob(\A)$.
  \end{romanlist}
\end{definition}

Observe that $I_{\X} = G(I_{\A})$ and $I_{\Y} = F(I_{\A})$, so $I_{\X} = I_{\Y}$ in $\obs$. We will denote this element simply by $I$. We expect $I$ to serve as the tensor unit in $\psh$, so we need the following property:

\begin{proposition}
  $I$ is an identity in $\obs$.
\end{proposition}

\begin{proof}
  Let $a = a_1 \ntns \ldots \ntns a_n \in \obs$, where each $a_i$ is in $\ob(\X) \sqcup \ob(\Y)$. Let $\C$ be the category that $a_1$ is in and $\D$ be the category that $a_n$ is in. Then
  \[
    I \ntns a = I_{\C} \ntns a_1 \ntns \ldots \ntns a_n = (I_{\C} \ntns a_1) \ntns \ldots \ntns a_n = a_1 \ntns \ldots \ntns a_n = a .
  \]
  Similarly,
  \[
    a \ntns I = a_1 \ntns \ldots \ntns a_n \ntns I_{\D} = a_1 \ntns \ldots \ntns (a_n \ntns I_{\D}) = a_1 \ntns \ldots \ntns a_n = a .
  \]
\end{proof}

\subsection{Morphisms}

The morphisms of $\psh$ are much more complex than the objects, being constructed via both the tensor product operation and the partial operation of composition. We must also account for braidings that do not arise as braidings from either $\X$ or $\Y$; that is, braidings $\sigma_{a b}$ where the expressions $a$ and $b$ in $\obs$ do not reduce to single objects in the same category. 

We wish to think of morphisms in $\psh$ as formal expressions in terms of the morphisms in $\X$ and $\Y$, additional braidings $\sigma_{a b}$, monoidal product, and composition, subject to the appropriate identifications to ensure that the symmetric monoidal category axioms are satisfied and that the desired pushout square commutes. Allowing arbitrary expressions, however, would result in malformed expressions, where morphisms are composed that should not be composable. Selecting only the well-formed expressions requires some work, and care must be taken to avoid accidentally identifying well-formed expressions with malformed ones.

For example, consider a monoidal category with objects $A, B, C$ such that $A \otimes B \neq A$ and morphisms $f$ with source $A \otimes B$, $g$ with target $C$, $h$ with target $A$, and $k$ with target $B \otimes C$. Then the expression $(f \otimes g) \circ (h \otimes k)$ is well-formed up to associativity, but blindly applying the usual distributivity formula yields the malformed expression $(f \circ h) \otimes (g \circ k)$, since $f$ and $h$ are not composable.

To this end, we begin with a highly simplified first approximation of $\mor(\psh)$ that will allow us to identify well-formed expressions and determine their sources and targets. Henceforth, we will refer to the data of an expression's well-formedness, along with its source and target if it is well-formed, as its \textbf{type}. 

First, we define an algebraic theory that allows us to define and manipulate the expressions we need without too much trouble. We will use the symbols $\ncmp$ and $\ntns$ instead of $\circ$ and $\otimes$, respectively, to avoid confusion with genuine composition and monoidal product:

\begin{definition}
  Let $\thr$ be the algebraic theory with two binary operations $\ncmp$ and $\ntns$ that are both associative.
\end{definition}

\begin{definition}
  Let $\pmaf$ be the free $\thr$-algebra generated by the following symbols:

  \begin{romanlist}
    \item a symbol $f$ for each $f \in \mor(\X) \sqcup \mor(\Y)$
    \item a symbol $\sigma_{a b}$ for each $a , b \in \obs$.
  \end{romanlist}
\end{definition}

\subsubsection{Expression Types}

We can now leverage the properties of $\thr$-algebras and their homomorphisms to easily determine the type of an expression. We will define a new $\thr$-algebra of $(\text{source},\text{target})$ pairs. The source and target of a tensor product of morphisms needs to be the respective tensor product of their sources and targets, and this does not depend on which morphisms are being tensored. 

The source and target of a composite of morphisms needs to be the familiar ``outermost" source and target of the constituent morphisms, but unlike with the monoidal product, not all morphisms are composable. Therefore, some expressions involving $\ncmp$ will result in malformed expressions. To deal with these, we introduce an element $\bot$ to represent them. Any expression with a malformed part is malformed, so we make $\bot$ absorptive.

\begin{definition}
  Let $\obt$ be the $\thr$-algebra on $\obs^2 \sqcup \{\bot\}$ with operations defined by:
  \begin{gather*}
    (a , b) \ntns (c , d) := (a \ntns c , b \ntns d) \\
    a \ntns \bot := \bot \ntns a := a \ncmp \bot := \bot \ncmp a := \bot \\
    (c , d) \ncmp (a , b) :=
    \begin{cases}
      (a , d) \quad & \text{if } b = c \text{ in } \obs \\
      \bot \quad & \text{if } b \neq c \text{ in } \obs .
    \end{cases}
  \end{gather*}
\end{definition}

\begin{proof}[Proof that $\obt$ is a $\thr$-algebra]
  Observe that $\obs^2$ is closed under $\ntns$, and $\ntns$ coincides on $\obs^2$ with the regular semigroup product operation, so it is associative on $\obs^2$. Outside of $\obs^2$, the only value of $\ntns$ is $\bot$, since the latter is absorptive. Thus, $\ntns$ is associative.

  Next, let $x , y , z \in \obt$. If $\bot \in \{x , y , z\}$, then $(x \ncmp y) \ncmp z = \bot = x \ncmp (y \ncmp z)$. Otherwise, there exist
  \[
    x_1 , x_2 , y_1 , y_2 , z_1 , z_2 \in \obs
  \]
 such that 
  \[
    x = (x_1 , x_2), \qquad y = (y_1 , y_2), \qquad z = (z_1 , z_2).
  \]
  The four cases are shown in the table below: 
  \begin{center}
    \begin{tabular}{c|cc|cc}
      & \multicolumn{2}{|c}{$y_2 = x_1$} & \multicolumn{2}{|c}{$y_2 \neq x_1$} \\
      & $z_2 = y_1$ & $z_2 \neq y_1$ & $z_2 = y_1$ & $z_2 \neq y_1$ \\
      \hline
      $x \ncmp y$ & $(y_1, x_2)$ & $(y_1, x_2)$ & $\bot$ & $\bot$ \\
      $y \ncmp z$ & $(z_1, y_2)$ & $\bot$ & $(z_1, y_2)$ & $\bot$ \\
      $(x \ncmp y) \ncmp z$ & $(z_1, x_2)$ & $\bot$ & $\bot$ & $\bot$ \\
      $x \ncmp (y \ncmp z)$ & $(z_1, x_2)$ & $\bot$ & $\bot$ & $\bot$
    \end{tabular}
  \end{center}
  Observe that the last two rows coincide, so $\ncmp$ is associative.
\end{proof}

We can now define the type of each expression using a $\thr$-algebra homomorphism, which is guaranteed to exist since $\pmaf$ is free:

\begin{definition}
  Let $\tp : \pmaf \to \obt$ be the $\thr$-algebra homomorphism defined by: 

  \begin{romanlist}
    \item $f \mapsto (a , b)$ for each $f : a \to b$ in $\mor(\X) \sqcup \mor(\Y)$
    \item $\sigma_{a b} \mapsto (a \ntns b , b \ntns a)$ for each $a , b \in \obs$.
  \end{romanlist}
\end{definition}

\subsubsection{Identity Morphisms}

There is one last component we need to define before we can state the necessary identifications on $\pmaf$: identity morphisms of arbitrary objects in $\psh$. We will refine our approximation slightly employing only the identifications needed to define these identity morphisms; these identifications correspond neatly to those made when defining $\obs$:

\begin{definition}
  Let $\pmas$ be the quotient of $\pmaf$ under the congruence relation generated by the identifications: 
  \begin{romanlist}
    \item $f \ntns g \sim f \otimes g$ for any $f , g \in \mor(\C)$, $\C \in \{\X , \Y\}$
    \item $F(f) \sim G(f)$ for any $f \in \mor(\A)$.
  \end{romanlist}
\end{definition}

\begin{definition}
  Let $\id_{-} : (\obs , \ntns) \to (\pmas, \ntns)$ be the semigroup homomorphism where $\id_a := \id_a^{\C}$ for any $a \in \ob(\C)$, $\C \in \{\X , \Y\}$.
\end{definition}

\begin{proof}[Proof that $\id_{-}$ is a semigroup homomorphism]
  Let $\C \in \{\X , \Y\}$ and $a , b \in \ob(\C)$. Then 
  \[
    \id_a \ntns \id_b = \id_a^{\C} \ntns \id_b^{\C} = \id_a^{\C} \otimes \id_b^{\C} = \id_{a \otimes b}^{\C} = \id_{a \otimes b} .
  \]
  For any $a \in \ob(\A)$, observe that 
  \[
    \id_{F(a)} = \id_{F(a)}^{\C} = \id_{G(a)}^{\C} = \id_{G(a)} .
    \qedhere
  \]
\end{proof}

We must also check that our new identifications are compatible with the assignment of types defined above:

\begin{proposition}
  The map $\tp : \pmaf \to \obt$ descends to the quotient $\pmas$.
\end{proposition}

\begin{proof}
  For any $\C \in \{\X , \Y\}$ and $f : a \to b$ and $g : c \to d$ in $\C$, 
  \[
    \tp(f \ntns g) = \tp(f) \ntns \tp(g) = (a , b) \ntns (c , d) = (a \ntns c , b \ntns d) = (a \otimes c , b \otimes d) = \tp(f \otimes g).
  \]
  Now, let $f : a \to b$ in $\A$. Then 
  \[
    \tp(F(f)) = (F(a), F(b)) = (G(a), G(b)) = \tp(G(f)).
  \]
  Therefore, $\tp$ induces a well-defined $\thr$-algebra homomorphism $\pmas \to \obt$.
\end{proof}

We need the identity morphism of any object $a$ in $\psh$ to be a well-formed expression with source and target $a$, and indeed this is the case:
 
\begin{proposition}
  For any $a \in \obs$, we have $\tp(\id_a) = (a , a)$.
\end{proposition}

\begin{proof}
  Observe that for any $\C \in \{\X , \Y\}$ and $a \in \ob(\C)$, 
  \[
    \tp(\id_a) = \tp(\id_a^c) = (a , a).
  \]
  Now, let $a , b \in \obs$ and assume that $\tp(\id_a) = (a,a)$ and $\tp(\id_b) = (b,b)$. Then 
  \[
    \tp(\id_{a \ntns b}) = \tp(\id_a \ntns \id_b) = \tp(\id_a) \ntns \tp(\id_b) = (a , a) \ntns (b , b) = (a \ntns b , a \ntns b).
  \]
  By induction, this concludes the proof. 
\end{proof}

\subsubsection{Algebra of Morphisms}

We are now ready to make all of the necessary identifications:

\begin{definition}
  Let $\pma$ be the quotient of $\pmas$ under the congruence relation generated by the following identifications, for any $\C \in \{\X, \Y\}$, $x, y \in \ob(\C)$, $\phi, \psi \in \mor(\C)$, $a, b, c, d \in \obs$, and $f, g, h, k \in \pmaf$: 
  \begin{romanlist}
    \item $\phi \ncmp \psi \sim \phi \circ \psi$ if $\phi$ and $\psi$ are composable in $\C$, \label{composition}
    \item $(f \ncmp g) \ntns (h \ncmp k) \sim (f \ntns h) \ncmp (g \ntns k)$ if $\tp(f \ncmp g), \tp(h \ncmp k) \neq \bot$, \label{distributivity}
    \item $\sigma_{x y} \sim \sigma_{x y}^{\C}$, \label{braidings}
    \item $\sigma_{a b} \ncmp \id_{a \ntns b} \sim \sigma_{a b}$, \label{identity}
    \item $\id_I \ntns \sigma_{a b} \sim \sigma_{a b} \ntns \id_I \sim \sigma_{a b}$, \label{unit}
    \item $(g \ntns f) \ncmp \sigma_{a b} \sim \sigma_{c d} \ncmp (f \ntns g)$ if $\tp(f) = (a , c)$ and $\tp(g) = (b , d)$, \label{naturality}
    \item $\sigma_{a (b \ntns c)} \sim (\id_b \ntns \sigma_{a c}) \ncmp (\sigma_{a b} \ntns \id_c)$, \label{hexagon}
    \item $\sigma_{b a} \ncmp \sigma_{a b} \sim \id_{a \ntns b}$. \label{symmetry}
  \end{romanlist}
\end{definition}

The first thing to do is check that none of these identifications interferes with our type assignments: 

\begin{proposition}
  The map $\tp : \pmaf \to \obt$ descends to the quotient $\pma$. 
\end{proposition}

\begin{proof}
  Let $\C \in \{\X , \Y\}$. Then for any $g : a \to b$ and $f : b \to c$ in $\C$, 
  \[
    \tp(f \ncmp g) = \tp(f) \ncmp \tp(g) = (b , c) \ncmp (a , b) = (a , c) = \tp(f \circ g).
  \]
  Next, for any $f , g , h , k$ in $\pmaf$ satisfying $\tp(f \ncmp g), \tp(h \ncmp k) \neq \bot$, we will show that 
  \[
    \tp((f \ncmp g) \ntns (h \ncmp k)) = \tp((f \ntns h) \ncmp (g \ntns k)).
  \]
  Since $\bot$ is absorptive, we have $\tp(\beta) \neq \bot$ for each $\beta \in \{f , g , h , k\}$ so there exist $\beta_1 , \beta_2 \in \obs$ such that $\tp(\beta) = (\beta_1 , \beta_2)$. Additionally, $f_1 = g_2$ and $h_1 = k_2$, so 
  \[
    \begin{split}
      \tp((f \ncmp g) \ntns (h \ncmp k)) & = (\tp(f) \ncmp \tp(g)) \ntns (\tp(h) \ncmp \tp(k)) \\
      & = ((f_1 , f_2) \ncmp (g_1 , g_2)) \ntns ((h_1 , h_2) \ncmp (k_1 , k_2)) \\
      & = (g_1 \ntns k_1 , f_2 \ntns h_2) \\
      & = (f_1 \ntns h_1 , f_2 \ntns h_2) \ncmp (g_1 \ntns k_1 , g_2 \ntns k_2) \\
      & = (\tp(f) \ntns \tp(h)) \ncmp (\tp(g) \ntns \tp(k)) \\
      & = \tp((f \ntns h) \ncmp (g \ntns k)).
    \end{split}
  \]
  Let $\C \in \{\X , \Y\}$ and $a , b \in \ob(\C)$. Then 
  \[
    \tp(\sigma_{a b}) = (a \ntns b , b \ntns a) = (a \otimes b , b \otimes a) = \tp(\sigma_{a b}^{\C}).
  \]
  For any $a , b \in \obs$,
  \[
      \tp(\sigma_{a b} \ncmp \id_{a \ntns b}) = (a \ntns b , b \ntns a) \ncmp (a \ntns b , a \ntns b) = (a \ntns b , b \ntns a) = \tp(\sigma_{a b}).
  \]
  Additionally,
  \[
    \tp(\sigma_{a b} \ntns \id_I) = (a \ntns b , b \ntns a) \ntns (I , I) = (a \ntns b , b \ntns a) = (I , I) \ntns (a \ntns b , b \ntns a) = \tp(\id_I \ntns \sigma_{a b}).
  \]
  and $(a \ntns b , b \ntns a) = \tp(\sigma_{a b})$. Now, let $a , b , c , d \in \obs$ and $f , g \in \pmaf$ satisfy $\tp(f) = (a , c)$ and $\tp(g) = (b , d)$. Then 
  \[
    \begin{split}
      \tp((g \ntns f) \ncmp \sigma_{a b}) & = (\tp(g) \ntns \tp(f)) \ncmp \tp(\sigma_{a b}) \\
      & = ((b , d) \ntns (a , c)) \ncmp (a \ntns b , b \ntns a) \\
      & = (a \ntns b , d \ntns c) \\
      & = (c \ntns d , d \ntns c) \ncmp ((a , c) \ntns (b , d)) \\
      & = \tp(\sigma_{c d}) \ncmp (\tp(f) \ntns \tp(g)) \\
      & = \tp(\sigma_{c d} \ncmp (f \ntns g)).
    \end{split}
  \]
  Next, let $a , b , c \in \obs$. Then
  \[
    \begin{split}
      \tp((\id_b \ntns \sigma_{a c}) \ncmp (\sigma_{a b} \ntns \id_c)) & = (\tp(\id_b) \ntns \tp(\sigma_{a c})) \ncmp (\tp(\sigma_{a b}) \ntns \tp(\id_c)) \\
      & = ((b , b) \ntns (a \ntns c , c \ntns a)) \ncmp ((a \ntns b , b \ntns a) \ntns (c , c)) \\
      & = (b \ntns a \ntns c , b \ntns c \ntns a) \ncmp (a \ntns b \ntns c , b \ntns a \ntns c) \\
      & = (a \ntns b \ntns c , b \ntns c \ntns a) = \tp(\sigma_{a (b \ntns c)}).
    \end{split}
  \]
  Additionally,
  \[
    \tp(\sigma_{b a} \ncmp \sigma_{a b}) = \tp(\sigma_{b a}) \ncmp \tp(\sigma_{a b}) = (b \ntns a , a \ntns b) \ncmp (a \ntns b , b \ntns a) = (a \ntns b , a \ntns b) = \tp(\id_{a \ntns b}).
  \]
  Thus, $\tp$ induces a well-defined $\thr$-algebra homomorphism $\pma \to \obt$. 
\end{proof}

\subsection{Construction of the category}

Equipped with our now-faithful description of morphism in $\psh$ we are finally ready to define it as a category:

\begin{definition}
Let $\psh$ be the category whose 
  \begin{itemize}
    \item set of objects is $\obs$,
    \item set of morphisms $a \to b$ is $\tp^{-1} (a , b) \subseteq \pma$ for any $a , b \in \obs$
    \item composition is $\ncmp$
    \item identity on $a$ is $\id_a$ for any $a \in \obs$.
  \end{itemize}
\end{definition}

\begin{proof}[Proof that $\psh$ is a category]
  Let $f : b \to c$ and $g : a \to b$ in $\psh$. Then 
  \[
    \tp(f \ncmp g) = \tp(f) \ncmp \tp(g) = (b , c) \ncmp (a , b) = (a , c),
  \]
  so $f \ncmp g : a \to c$. Thus, composition is well-defined. Notice that composition in $\psh$ is associative by definition.

  Next, we will use induction to show that for any $f : a \to b$ in $\psh$, we have $\id_b \ncmp f = f \ncmp \id_a = f$. For any $\C \in \{\X , \Y\}$ and $f : a \to b$ in $\C$, we have $f : a \to b$ in $\psh$ and 
  \[
    \id_b \ncmp f = \id_b^{\C} \ncmp f = \id_b^{\C} \circ f = f = f \circ \id_a^{\C} = f \ncmp \id_a^{\C} = f \ncmp \id_a .
  \]
  Additionally, for any $a , b \in \obs$, 
  \[
    \id_{b \ntns a} \ncmp \sigma_{a b} = (\id_b \ntns \id_a) \ncmp \sigma_{a b} = \sigma_{a b} \ncmp (\id_a \ntns \id_b) = \sigma_{a b} \ncmp \id_{a \ntns b} = \sigma_{a b} .
  \]
  Next, for any $f : a \to b$ and $g : b \to c$ in $\psh$ such that $f \ncmp \id_a = f$ and $\id_c \ncmp g = g$, we have $\tp(g \ncmp f) = (a , c)$. Thus, 
  \[
    \id_c \ncmp g \ncmp f = g \ncmp f = g \ncmp f \ncmp \id_a .
  \]
  Finally, for any $f : a \to b$ and $g : c \to d$ in $\psh$ such that $\id_b \ncmp f = f \ncmp \id_a = f$ and $\id_d \ncmp g = g \ncmp \id_c = g$, we have 
  \[
    \tp(\id_b \ncmp f) = \tp(\id_b) \ncmp \tp(f) = (b , b) \ncmp (a , b) = (a , b) \neq \bot
  \]
  and 
  \[
    \tp(\id_d \ncmp g) = \tp(\id_d) \ncmp \tp(g) = (d , d) \ncmp (a , d) = (a , d) \neq \bot .
  \]
  Therefore, 
  \[
    \id_{b \ntns d} \ncmp (f \ntns g) = (\id_b \ntns \id_d) \ncmp (f \ntns g) = (\id_b \ncmp f) \ntns (\id_d \ncmp g) = f \ntns g .
  \]
  Similarly, $\tp(f \ncmp \id_a) = (a , b) \neq \bot$ and $\tp(g \ncmp \id_c) = (c , d) \neq \bot$, so 
  \[
    (f \ntns g) \ncmp \id_{a \ntns c} = (f \ntns g) \ncmp (\id_a \ntns \id_c) = (f \ncmp \id_a) \ntns (g \ncmp \id_c) = f \ntns g .
  \]
  Hence, we may conclude that for each $a \in \obs$, the morphism $\id_a$ is an identity in $\psh$. 
\end{proof}

\subsection{Monoidal structure}

We now verify that we have the desired monoidal structure on $\psh$. First, we need to check that the monoidal product is functorial:

\begin{proposition}
  The operation $\ntns$ on $\obs$ and $\pma$ forms a functor $(\psh)^2 \to \psh$. 
\end{proposition}

\begin{proof}
  Let $f : a \to b$ and $g : c \to d$ in $\psh$. Then 
  \[
    \tp(f \ntns g) = \tp(f) \ntns \tp(g) = (a , b) \ntns (c , d) = (a \ntns c , b \ntns d),
  \]
  so $f \ntns g : a \ntns c \to b \ntns d$. Observe that $\ntns$ preserves composition by identification \ref{distributivity} above, and it preserves identities since $\id_{-}$ is defined to be a homomorphism $\obs \to \pmas$ under $\ntns$. 
\end{proof}

Now, we verify that the axioms of a strict monoidal category hold.

\begin{proposition}
  The functor $\ntns : (\psh)^2 \to \psh$ makes $\psh$ into a strict monoidal category with monoidal unit $I$. 
\end{proposition}

\begin{proof}
  By definition, $\ntns$ is associative on both objects and morphisms. Additionally, $I$ is an identity for objects. We will use induction to show that for all $f \in \mor(\psh)$, we have $I \ntns f = f \ntns \id_I = f$. Let $\C \in \{\X , \Y\}$ and $f \in \mor(\C)$. Then 
  \[
    \id_I \ntns f = \id_{I_{\C}} \ntns f = \id_{I_{\C}} \otimes f = f = f \otimes \id_{I_{\C}} = f \ntns \id_{I_{\C}} = f \ntns \id_I .
  \]
  For any $a , b \in \obs$, we have $\id_I \ntns \sigma_{a b} = \sigma_{a b} \ntns \id_I = \sigma_{a b}$ by identification \ref{unit}. Next, let $f , g \in \mor(\psh)$ satisfy $\id_I \ntns f = f \ntns \id_I = f$ and $\id_I \ntns g = g \ntns \id_I = g$. Then 
  \[
    \id_I \ntns (f \ntns g) = (\id_I \ntns f) \ntns g = f \ntns g = f \ntns (g \ntns \id_I) = (f \ntns g) \ntns \id_I .
  \]
  Finally, let $a , b , c \in \ob(\psh)$, $f : a \to b$, and $g : b \to c$ satisfy $\id_I \ntns f = f \ntns \id_I = f$ and $\id_I \ntns g = g \ntns \id_I = g$. Then 
  \begin{multline*}
      \id_I \ntns (g \ncmp f) = (\id_I \ncmp \id_I) \ntns (g \ncmp f) 
      = (\id_I \ntns g) \ncmp (\id_I \ntns f) 
      = g \ncmp f \\
      = (g \ntns \id_I) \ncmp (f \ntns \id_I) 
      = (g \ncmp f) \ntns (\id_I \ncmp \id_I) 
      = (g \ncmp f) \ntns \id_I ,
  \end{multline*}
  which completes the inductive step. Thus, $\id_I$ is an identity with respect to $\ntns$ for morphisms. 
\end{proof}

Next, we show that the formal braiding expressions $\sigma_{a b}$ form a genuine symmetric braiding; this is straightforward given the identifications made when defining $\pma$.

\begin{proposition}
  The family $\sigma := {\sigma_{a b} }$ for ${a , b \in \obs}$ is a symmetric braiding on $\psh$. 
\end{proposition}

\begin{proof}
  Naturality of $\sigma$ follows directly from identification \ref{naturality}. The hexagon identity follows directly from identification \ref{hexagon}, and the symmetry of $\sigma$ follows directly from identification \ref{symmetry}. 
\end{proof}

This completes the construction of the strict symmetric monoidal category $\psh$.

\subsection{Universal property of the pushout}

We turn our attention to verifying that $\psh$ is the pushout of the diagram \ref{eq:diagram-for-pushout}.

\subsubsection{Construction of the universal cocone} \label{sec:cocone}

First, we show that the operations of interpreting objects and morphisms in $\X$ and $\Y$ as objects and morphisms in $\psh$ form a cocone from the given diagram to $\psh$:

\begin{proposition} \label{prop:cocone}
  Let $\C \in \{\X , \Y\}$. The canonical maps $\ob(\C) \to \obs$ and $\mor(\C) \to \pma$ form the object and morphism parts of a strict symmetric monoidal functor $\iota_{\C} : \C \to \psh$. 
\end{proposition}

\begin{proof}
  We will first show that $\iota_{\C}$ is functorial. Let $f : a \to b$ in $\C$. By definition, 
  \[
    \tp(\iota_{\C} (f)) = (a , b) = (\iota_{\C} (a), \iota_{\C} (b)).
  \]
  Additionally, for any $g : b \to c$ in $\C$, 
  \[
    \iota_{\C} (f \circ g) = f \circ g = f \ncmp g = \iota_{\C} (f) \ncmp \iota_{\C} (g).
  \]
  For any $a \in \ob(\C)$, we also have 
  \[
    \iota_{\C} (\id_a^{\C}) = \id_a^{\C} = \id_a ,
  \]
  so $\iota_{\C}$ is functorial. 
  
  To show that $\iota_\C$ is a strict monoidal functor, first note that for any $a , b \in \ob(\C)$, 
  \[
    \iota_{\C} (a \otimes b) = a \otimes b = a \ntns b = \iota_{\C} (a) \ntns \iota_{\C} (b).
  \]
  Next, observe that regardless of the value of $\C$, we have 
  \[
    \iota_{\C} (I_{\C}) = I_{\C} = I
  \]
  in $\psh$. For any $f , g \in \mor(\C)$, observe that 
  \[
    \iota_{\C} (f \otimes g) = f \otimes g = f \ntns g = \iota_{\C} (f) \ntns \iota_{\C} (g).
  \]
  Thus, $\iota_{\C}$ preserves the monoidal operation and unit strictly. Finally, for any $a , b \in \ob(\C)$, note that 
  ${\iota_{\C} (\sigma_{a b}^{\C}) = \sigma_{a b}^{\C} = \sigma_{a b}}$, so $\iota_{\C}$ preserves the braiding on $\C$. Thus, $\iota_{\C}$ is a symmetric monoidal functor. 
\end{proof}

Below, we denote by $\smc$ the category of strict symmetric monoidal categories with strong symmetric monoidal functors. It is straightforward to check, using the definitions of $\obs$ and $\pmas$ together with Proposition \ref{prop:cocone}, that the diagram \[
    \begin{tikzcd}
      \A && \Y \\
      \\
      \X && \psh
      \arrow["F", from=1-1, to=1-3]
      \arrow["G"', from=1-1, to=3-1]
      \arrow["{\iota_\Y}", from=1-3, to=3-3]
      \arrow["{\iota_\X}"', from=3-1, to=3-3]
    \end{tikzcd}
  \] commutes in $\smc$.

\subsubsection{Universality of the Cocone} Finally, we verify that the cocone to $\psh$ constructed in \ref{sec:cocone} is initial. Since the morphisms of $\psh$ are defined entirely in terms of semigroups and $\thr$-algebras, it will be much easier to define the necessary functor out of $\psh$ if its target can be described in terms of semigroups and $\thr$-algebras as well. The objects of any strict symmetric monoidal category already form a semigroup; we will show that the morphisms of any such category in fact form a $\thr$-algebra. Like before, we need to add an extra element $\bot$ to handle malformed expressions:

\begin{definition}
  For any strict symmetric monoidal category $\W$, let $\cta(\W)$ be the $\thr$-algebra on $\mor(\W) \sqcup \{\bot\}$ with operations defined by: 
  \begin{gather*}
    f \ntns g :=
    \begin{cases}
      f \otimes g \quad & \text{if } f, g \neq \bot \\
      \bot \quad & \text{if } \bot \in \{f, g\}
    \end{cases} \\
    f \ncmp g :=
    \begin{cases}
      f \circ g \quad & \text{if } f, g \neq \bot \text{ and } f \text{ and } g \text{ are composable} \\
      \bot \quad & \text{if } \bot \in \{f, g\} \text{ or } f \text{ and } g \text{ are not composable.}
    \end{cases}
  \end{gather*}
\end{definition}

To see that $\cta(\W)$ is a $\thr$-algebra, first observe that $\mor(\W)$ is closed under $\ntns$, and $\ntns$ coincides on $\mor(\W)$ with $\otimes$, which is strictly associative. Thus, $\ntns$ is associative on $\mor(\W)$. Outside of $\mor(\W)$, the only value of $\ntns$ is $\bot$, since the latter is absorptive, so $\ntns$ is associative.

  Next, let $f, g, h \in \cta(\W)$. If $\bot \in \{f, g, h\}$, then $(f \ncmp g) \ncmp h = \bot = f \ncmp (g \ncmp h)$. Otherwise, $f , g , h \in \mor(\W)$, so there exist $f_1 , f_2 , g_1 , g_2 , h_1 , h_2 \in \ob(\W)$ such that 
  \[
    f : f_1 \to f_2 , \qquad g : g_1 \to g_2 , \qquad h : h_1 \to h_2 .
  \]
  The four cases are shown in the table below: 

  \begin{center}
    \begin{tabular}{c|cc|cc}
      & \multicolumn{2}{|c}{$g_2 = f_1$} & \multicolumn{2}{|c}{$g_2 \neq f_1$} \\
      & $h_2 = g_1$ & $h_2 \neq g_1$ & $h_2 = g_1$ & $h_2 \neq g_1$ \\
      \hline
      $f \ncmp g$ & $f \circ g$ & $f \circ g$ & $\bot$ & $\bot$ \\
      $g \ncmp h$ & $g \circ h$ & $\bot$ & $g \circ h$ & $\bot$ \\
      $(f \ncmp g) \ncmp h$ & $(f \circ g) \circ h$ & $\bot$ & $\bot$ & $\bot$ \\
      $f \ncmp (g \ncmp h)$ & $f \circ (g \circ h)$ & $\bot$ & $\bot$ & $\bot$
    \end{tabular}
  \end{center}
  Observe that the last two rows coincide by the associativity of $\circ$, so $\ncmp$ is associative.

Finally, we have all the necessary machinery to prove that $\psh$ is the desired pushout:

\begin{theorem} \label{thm:pushout}
  $(\psh , \iota_{\X} , \iota_{\Y})$ is the pushout of the diagram $\X \xleftarrow{G} \A \xrightarrow{F} \Y$ in the category $\smc$.
\end{theorem}

\begin{proof}
  Given a commutative diagram of the following form in $\smc$: 
  \[
    \begin{tikzcd}
      \A && \Y \\
      \\
      \X && \W
      \arrow["F", from=1-1, to=1-3]
      \arrow["G"', from=1-1, to=3-1]
      \arrow["{\kappa_\Y}", from=1-3, to=3-3]
      \arrow["{\kappa_\X}"', from=3-1, to=3-3]
    \end{tikzcd}
  \]
  let $h_1 : (\obs , \ntns) \to (\ob(\W), \otimes)$ be the semigroup homomorphism defined by $h_1 (a) := \kappa_{\C} (a)$ for every $\C \in \{\X , \Y\}$ and $a \in \ob(\C)$, and let $h_2 : \pmas \to \cta(\W)$ be the $\thr$-algebra homomorphism defined by 
  \begin{itemize}
    \item $h_2 (f) := \kappa_{\C} (f)$ for every $\C \in \{\X , \Y\}$ and $f \in \mor(\C)$
    \item $h_2 (\sigma_{a b}) := \sigma_{h_1 (a) h_1 (b)}^{\W}$ for each $a , b \in \obs$.
  \end{itemize}
  We will first show that $h_1$ and $h_2$ are well-defined. Let 
  \[
    \C \in \{\X , \Y\}
  \]
  and $a , b \in \ob(C)$. Then 
  \[
    h_1 (a \ntns b) = h_1 (a) \ntns h_1 (b) = \kappa_{\C} (a) \otimes \kappa_{\C} (b) = \kappa_{\C} (a \otimes b) = h_1 (a \otimes b).
  \]
  Additionally, for any $a \in \ob(\A)$, 
  \[
    h_1 (G(a)) = \kappa_{\X} (G(a)) = \kappa_{\Y} (F(a)) = h_1 (F(a)).
  \]
  Thus, $h_1$ is well-defined. Let $\C \in \{\X , \Y\}$ and $f , g \in \mor(\C)$. Then 
  \[
    h_2 (f \ntns g) = h_2 (f) \ntns h_2 (g) = \kappa_{\C} (f) \ntns \kappa_{\C} (g) = \kappa_{\C} (f) \otimes \kappa_{\C} (g) = \kappa_{\C} (f \otimes g) = h_2 (f \otimes g).
  \]
  For any $f \in \mor(\A)$, 
  \[
    h_2 (G(f)) = \kappa_{\X} (G(f)) = \kappa_{\Y} (F(f)) = h_2 (F(f)).
  \]
  Therefore, $h_2$ is well-defined.

  Now, we will use induction to show that for any $a \in \obs$, we have $h_2 (\id_a) = \id_{h_1 (a)}^{\W}$. For any $\C \in \{\X , \Y\}$ and any $a \in \ob(\C)$, we 
  \[
    h_2 (\id_a) = h_2 (\id_a^{\C}) = \kappa_{\C} (\id_a^{\C}) = \id_{\kappa_{\C} (a)}^{\W} = \id_{h_1 (a)}^{\W} .
  \]
  For any $a , b \in \obs$ such that $h_2 (\id_a) = \id_{h_1 (a)}^{\W}$ and $h_2 (\id_b) = \id_{h_1 (b)}^{\W}$, 
  \[
    h_2 (\id_{a \ntns b}) = h_2 (\id_a) \ntns h_2 (\id_b) = \id_{h_1 (a)}^{\W} \ntns \id_{h_1 (b)}^{\W} = \id_{h_1 (a)}^{\W} \otimes \id_{h_1 (b)}^{\W} = \id_{h_1 (a \ntns b)}^{\W}.
  \]
  This concludes the inductive step.

  Now, we will use induction to show that for any $f \in \pmas$ such that $\tp(f) = (a , b)$ for some $a , b \in \obs$, the element $h_2 (f) \in \cta(\W)$ is a morphism $h_1 (a) \to h_1 (b)$ in $\W$. Note that as a consequence, we then have that for any $f , g \in \pmas$, we have $h_2 (f \ntns g) = h_2 (f) \otimes h_2 (g)$ if $\tp(f), \tp(g) \neq \bot$ and $h_2 (f \ncmp g) = h_2 (f) \circ h_2 (g)$ if $\tp(f \ncmp g) \neq \bot$.

  First, for any $\C \in \{\X , \Y\}$ and $f : a \to b$ in $\C$, we have $h_2 (f) = \kappa_{\C} (f) : \kappa_{\C} (a) \to \kappa_{\C} (b)$ in $\W$ and $\tp(f) = (a , b)$ in $\pma$. Since $a , b \in \ob(\C)$, we have $h_1 (a) = \kappa_{\C} (a)$ and $h_1 (b) = \kappa_{\C} (b)$, so $h_2 (f): h_1 (a) \to h_1 (b)$. Next, for any $a , b \in \obs$, observe that 
  \[
    h_2 (\sigma_{a b}) = \sigma_{h_1 (a) h_1 (b)}^{\W} : h_1 (a) \otimes h_1 (b) \to h_1 (b) \otimes h_1 (a)
  \]
  in $\W$. Since $h_1$ is a semigroup homomorphism, this means that $h_2 (\sigma_{a b})$ is a morphism in $\W$ with source $h_1 (a \ntns b)$ and target $h_1 (b \ntns a)$. Recall that $\tp(\sigma_{a b}) = (a \ntns b , b \ntns a)$, so $\sigma_{a b}$ has the desired property.

  Now, let $f , g \in \pmas$ have the desired property and satisfy $\tp(f \ntns g) \neq \bot$. Then since $\bot$ is absorptive, there exist $a , b , c , d \in \obs$ such that $\tp(f) = (a , b)$ and $\tp(g) = (c , d)$, meaning that $h_2 (f): h_1 (a) \to h_1 (b)$ and $h_2 (g) : h_1 (c) \to h_1 (d)$ in $\W$. Thus, 
  \[
    h_2 (f \ntns g) = h_2 (f) \ntns h_2 (g) = h_2 (f) \otimes h_2 (g) : h_1 (a) \otimes h_1 (c) \to h_1 (b) \otimes h_1 (d).
  \]
  Since $h_1$ is a semigroup homomorphism, this means that $h_2 (f \ntns g): h_1 (a \ntns c) \to h_1 (b \ntns d)$ in $\W$. Observe that $\tp(f \ntns g) = (a \ntns c , b \ntns d)$, so $f \ntns g$ has the desired property.

  Letting $f , g \in \pmas$ instead satisfy $\tp(f \ncmp g) \neq \bot$, we may conclude that there exist $a , b , c \in \obs$ such that $\tp(f) = (b , c)$ and $\tp(g) = (a , b)$, so $h_2 (f): h_1 (b) \to h_1 (c)$ and $h_2 (g) : h_1 (a) \to h_1 (b)$ in $\W$. Thus, 
  \[
    h_2 (f \ncmp g) = h_2 (f) \ncmp h_2 (g) = h_2 (f) \circ h_2 (g) : h_1 (a) \to h_1 (c).
  \]
  Since $\tp(f \ncmp g) = (a , c)$, this means that $f$ has the desired property, completing the induction. 

  Next, we will show that $h_2$ descends to the quotient $\pma$. We begin by verifying that $h_2$ respects identification \ref{composition}; for any $\C \in \{\X , \Y\}$ and composable $f , g \in \mor(\C)$, we have $\tp(f \ncmp g) \neq \bot$, so 
  \[
    h_2 (f \ncmp g) = h_2 (f) \circ h_2 (g) = \kappa_{\C} (f) \circ \kappa_{\C} (g) = \kappa_{\C} (f \circ g) = h_2 (f \circ g).
  \]
  We will now verify that $h_2$ respects identification \ref{distributivity}; for any $f , g , \ell , m \in \pmaf$ such that $\tp(f \ncmp g), \tp(\ell \ncmp m) \neq \bot$, we have $\tp((f \ntns \ell) \ncmp (g \ntns m)) \neq \bot$. Therefore, 
  \[
    \begin{split}
      h_2 ((f \ncmp g) \ntns (\ell \ncmp m)) & = (h_2 (f) \circ h_2 (g)) \otimes (h_2 (\ell) \circ h_2 (m)) \\
      & = ( h_2 (f) \otimes h_2 (\ell)) \circ (h_2 (g) \otimes h_2 (m)) \\
      & = h_2 ((f \ntns \ell) \ncmp (g \ntns m)).
    \end{split}
  \]
  Next, we verify that $h_2$ respects identification \ref{braidings}; for any $\C \in \{\X , \Y\}$ and $a , b \in \ob(\C)$, 
  \[
    h_2 (\sigma_{a b}) = \sigma_{h_1 (a) h_1 (b)}^{\W} = \sigma_{\kappa_{\C} (a) \kappa_{\C} (b)}^{\W} = \kappa_{\C} (\sigma_{a b})
  \]
  since $\kappa_{\C}$ is symmetric and strict. Additionally, $\tp(\sigma_{a b} \ncmp \id_{a \ntns b}) \neq \bot$, so 
  \[
    h_2 (\sigma_{a b} \ncmp \id_{a \ntns b}) = \sigma_{h_1 (a) h_1 (b)}^{\W} \circ \id_{h_1 (a \ntns b)}^{\W} = \sigma_{h_1 (a) h_1 (b)}^{\W} = h_2 (\sigma_{a b}).
  \]
  Thus, $h_2$ respects identification \ref{identity}. Observe as well that 
  \[
    h_2 (\id_I) = h_2 (\id_{I_{\X}}) = \kappa_{\X} (\id_{I_{\X}}) = \id_{I_{\W}}^{\W} ,
  \]
  so for any such $a$ and $b$, 
  \[
    h_2 (\id_I \ntns \sigma_{a , b}) = h_2 (\id_I) \otimes h_2 (\sigma_{a , b}) = \id_{I_{\W}}^{\W} \otimes \sigma_{a , b}^{\W} = \sigma_{a , b}^{\W} = h_2 (\sigma_{a , b})
  \]
  and 
  \[
    h_2 (\sigma_{a , b} \ntns \id_I) = h_2 (\sigma_{a , b}) \otimes h_2 (\id_I) = \sigma_{a , b}^{\W} \otimes \id_{I_{\W}}^{\W} = \sigma_{a , b}^{\W} = h_2 (\sigma_{a , b}).
  \]
  This shows that $h_2$ respects identification \ref{unit}. Now, we turn to identity \ref{naturality}; for any $a , b , c , d \in \obs$ and any $f , g \in \pmaf$ such that $\tp(f) = (a , c)$ and $\tp(g) = (b , d)$, we have 
  \[
    \tp((g \ntns f) \ncmp \sigma_{a b}), \tp(\sigma_{c d} \ncmp (f \ntns g)) \neq \bot .
  \]
  Additionally, $h_2 (f): h_1 (a) \to h_1 (c)$ and $h_2 (g): h_1 (b) \to h_1 (d)$ in $\W$, so by the naturality of $\sigma^{\W}$,
  \[
    h_2 ((g \ntns f) \ncmp \sigma_{a b}) = (h_2 (g) \otimes h_2 (f)) \circ \sigma_{h_1 (a) h_1 (b)}^{\W} = \sigma_{h_1 (c) h_1 (d)}^{\W} \circ (h_2 (f) \ntns h_2 (g)) = h_2 (\sigma_{c d} \ncmp (f \ntns g)). 
  \]
  We will show that $h_2$ also respects identification \ref{hexagon}; for any $a , b , c \in \obs$, we have $\tp((\id_b \ntns \sigma_{a c}) \ncmp (\sigma_{a b} \ntns \id_c)) \neq \bot$, so
  \[
    \begin{split}
      h_2 (\sigma_{a (b \ntns c)}) & = \sigma_{h_1 (a) (h_1 (b) \otimes h_1 (c))}^{\W} \\
      & = (\id_{h_1 (b)}^{\W} \otimes \sigma_{h_1 (a) h_1 (c)}^{\W}) \circ (\sigma_{h_1 (a) h_1 (b)}^{\W} \otimes \id_{h_1 (c)}^{\W}) \\
      & = h_2 ((\id_b \ntns \sigma_{a c}) \ncmp (\sigma_{a b} \ntns \id_c))
    \end{split}
  \]
  Finally, we will show that $h_2$ respects identification \ref{symmetry}; for any $a , b \in \obs$, we have $\tp(\sigma_{b a} \ncmp \sigma_{a b}) \neq \bot$, so
  \[
    h_2 (\sigma_{b a} \ncmp \sigma_{a b}) = \sigma_{h_1 (b) h_1 (a)}^{\W} \circ \sigma_{h_1 (a) h_1 (b)}^{\W} = \id_{h_1 (a \ntns b)}^{\W} = h_2 (\id_{a \ntns b}).
  \]
  Thus, $h_2$ induces a well-defined $\thr$-algebra homomorphism $\pma \to \cta(\W)$.

  Now, since $h_2$ takes morphisms in $\psh$ to morphisms in $\W$ with the expected sources and targets given by $h_1$, takes $\ntns$ to $\otimes$ and $\ncmp$ to $\circ$ when appropriate, and preserves identity morphisms, we may define a functor $H : \psh \to \W$ whose action on objects is given by $h_1$ and whose action on morphisms is given by $h_2$. We will show that $H$ is a strict monoidal functor. Since $h_1$ and $h_2$ are strictly associative, $H$ is strictly associative on both objects and morphisms. To see that $H$ is unital, let $a \in \obs$ and observe that 
  \[
    H(I) = H(I_{\X}) = \kappa_{\X} (I_{\X}) = I_{\W} .
  \]
  Additionally, $H$ preserves the braiding by the definition of $h_2$, so it is a symmetric monoidal functor. For $\C \in \{\X , \Y\}$, the fact that $H \circ \iota_{\C} = \kappa_{\C}$ follows directly from the definition of $H$.

  Finally, we will show that $H$ is unique. Let $L : \psh \to \W$ satisfy $L \circ \iota_{\C} = \kappa_{\C}$ for each $\C \in \{\X , \Y\}$. Then for any $a \in \ob(\C)$, 
  \[
    H(a) = H(\iota_{\C} (a)) = \kappa_{\C} (a) = L(\iota_{\C} (a)) = L(a).
  \]
  Thus, $H$ and $L$ agree on objects. For any $\C \in \{\X , \Y\}$ and $f \in \mor(\C)$, 
  \[
    H(f) = H(\iota_{\C} (f)) = \kappa_{\C} (f) = L(\iota_{\C} (f)) = L(f).
  \]
  Additionally, for any $a , b \in \obs$, 
  \[
    H(\sigma_{a b}) = \sigma_{H(a) H(b)}^{\W} = L(\sigma_{a b})
  \]
  since $H$ and $L$ are both symmetric. Thus, $H = L$. 
\end{proof}

\section{Preservation of central idempotents} \label{sec:central-idempotents}

To conclude, we would like to show that the central idempotent lattices $\zi(\X)$ and $\zi(\Y)$ are sublattices of $\zi(\psh)$. Indeed, we will show that any strong braided monoidal functor $K : \C \to \D$ preserves central idempotents in a sense defined below in \ref{def:stronglypreserves}, so that $\zi(K\C) \subset \zi(\D)$. In particular, this inclusion will hold for the canonical embeddings $\X \to \psh$ and $\Y \to \psh$.

We first verify that the property of being a central idempotent is preserved under precomposition by isomorphisms.

\begin{proposition}\label{prop:equivalent-central-idempotents}
  Let $\C$ be a monoidal category, $u: U \to I$ be a central idempotent in $\C$, and $f: V \to U$ be an isomorphism. Then $v := u \circ f$ is a central idempotent equivalent to $u$.
\end{proposition}

\begin{proof}
  Set $\sigma_{V A} := (A \otimes f^{-1}) \circ \sigma_{U A} \circ (f \otimes A)$ for each $A \in \ob(\C)$. We will show that $\sigma_V$ is a half-braiding. For any $g: A \to B$ in $\C$, the following diagram commutes, so $\sigma_V$ is natural.
  \[
    \begin{tikzcd}
  	{V \otimes A} &&&&&& {V \otimes B} \\
  	\\
  	&& {U \otimes A} && {U \otimes B} \\
  	\\
  	&& {A \otimes U} && {B \otimes U} \\
  	\\
  	{A \otimes V} &&&&&& {B \otimes V}
  	\arrow["{g \otimes A}", from=1-1, to=1-7]
  	\arrow["{f \otimes A}"{description}, from=1-1, to=3-3]
  	\arrow["{\sigma_{V A}}", from=1-1, to=7-1]
  	\arrow["{f \otimes B}"{description}, from=1-7, to=3-5]
  	\arrow["{\sigma_{V B}}", from=1-7, to=7-7]
  	\arrow["{U \otimes g}", from=3-3, to=3-5]
  	\arrow["{\sigma_{U A}}", from=3-3, to=5-3]
  	\arrow["{\sigma_{U B}}", from=3-5, to=5-5]
  	\arrow["{g \otimes U}", from=5-3, to=5-5]
  	\arrow["{A \otimes f^{-1}}"{description}, from=5-3, to=7-1]
  	\arrow["{B \otimes f^{-1}}"{description}, from=5-5, to=7-7]
  	\arrow["{g \otimes V}", from=7-1, to=7-7]
    \end{tikzcd}
  \]
  Additionally, for any $A, B \in \ob(\C)$, the following diagram commutes, \\
  \adjustbox{scale=.9}{
    \begin{tikzcd}
      {V \otimes (A \otimes B)} &&&& {(V \otimes A) \otimes B} &&&& {(A \otimes V) \otimes B} \\
      \\
      && {U \otimes (A \otimes B)} && {(U \otimes A) \otimes B} && {(A \otimes U) \otimes B} \\
      && {(A \otimes B) \otimes U} && {A \otimes (B \otimes U)} && {A \otimes (U \otimes B)} \\
      \\
      {(A \otimes B) \otimes V} &&&& {A \otimes (B \otimes V)} &&&& {A \otimes (V \otimes B)}
      \arrow["{{\alpha^{-1}_{V A B}}}", from=1-1, to=1-5]
      \arrow["{{f \otimes (A \otimes B)}}"{description}, from=1-1, to=3-3]
      \arrow["{{\sigma_{V (A \otimes B)}}}", from=1-1, to=6-1]
      \arrow["{{\sigma_{V A} \otimes B}}", from=1-5, to=1-9]
      \arrow["{{(f \otimes A) \otimes B}}", from=1-5, to=3-5]
      \arrow["{{\alpha_{A V B}}}", from=1-9, to=6-9]
      \arrow["{{\alpha^{-1}_{U A B}}}", from=3-3, to=3-5]
      \arrow["{{\sigma_{U (A \otimes B)}}}", from=3-3, to=4-3]
      \arrow["{{\sigma_{U A} \otimes B}}", from=3-5, to=3-7]
      \arrow["{{(A \otimes f^{-1}) \otimes B}}"{description}, from=3-7, to=1-9]
      \arrow["{{\alpha_{A U B}}}", from=3-7, to=4-7]
      \arrow["{{(A \otimes B) \otimes f^{-1}}}"{description}, from=4-3, to=6-1]
      \arrow["{{\alpha^{-1}_{A B U}}}"', from=4-5, to=4-3]
      \arrow["{{A \otimes (B \otimes f^{-1})}}", from=4-5, to=6-5]
      \arrow["{{A \otimes \sigma_{U B}}}"', from=4-7, to=4-5]
      \arrow["{{\alpha^{-1}_{A B V}}}"', from=6-5, to=6-1]
      \arrow["{{A \otimes (f \otimes B)}}"{description}, from=6-9, to=4-7]
      \arrow["{{A \otimes \sigma_{V B}}}"', from=6-9, to=6-5]
    \end{tikzcd}
  } \\
  so $\sigma_V$ is a half-braiding.

  Next, we observe that $f$ and $f^{-1}$ are morphisms in $Z(\C)$; for any $A \in \ob(\C)$,
  \[
    (A \otimes f) \circ \sigma_{V A}
    = (A \otimes f) \circ (A \otimes f^{-1}) \circ \sigma_{U A} \circ (f \otimes A)
    = \sigma_{U A} \circ (f \otimes A) \\
  \]
  and
  \[
    (A \otimes f^{-1}) \circ \sigma_{U A}
    = (A \otimes f^{-1}) \circ \sigma_{U A} \circ (A \otimes f) \circ (f^{-1} \otimes A)
    = \sigma_{V A} \circ (f^{-1} \otimes A).
  \]
  Thus, $v$ is also a morphism in $Z(\C)$.

  Finally, observe that the following diagram commutes:
  \[
    \begin{tikzcd}
      {V \otimes V} &&&& {V \otimes I} \\
      \\
      && {U \otimes V} && {U \otimes I} \\
      \\
      {I \otimes V} && {I \otimes U} && {U \otimes U}
      \arrow["{{V \otimes v}}", from=1-1, to=1-5]
      \arrow["{{f \otimes V}}"{description}, from=1-1, to=3-3]
      \arrow["{{v \otimes V}}", from=1-1, to=5-1]
      \arrow["{{f \otimes I}}", from=1-5, to=3-5]
      \arrow["{{U \otimes v}}", from=3-3, to=3-5]
      \arrow["{{u \otimes V}}"{description}, from=3-3, to=5-1]
      \arrow["{{U \otimes f}}"{description}, from=3-3, to=5-5]
      \arrow["{{I \otimes f}}", from=5-1, to=5-3]
      \arrow["{{U \otimes u}}"', from=5-5, to=3-5]
      \arrow["{{u \otimes U}}"', from=5-5, to=5-3]
    \end{tikzcd}
  \]
  This means that the topmost and leftmost trapezoids in the following diagram commute, so the diagram as a whole commutes:
  \[
    \begin{tikzcd}
      {V \otimes V} &&&&&& {V \otimes I} \\
      && {U \otimes U} && {U \otimes I} \\
      \\
      && {I \otimes U} && U \\
      {I \otimes V} &&&&&& V \\
      &&&& {}
      \arrow["{{V \otimes v}}", from=1-1, to=1-7]
      \arrow["{{f \otimes f}}"{description}, from=1-1, to=2-3]
      \arrow["{{v \otimes V}}", from=1-1, to=5-1]
      \arrow["{{f \otimes I}}"{description}, from=1-7, to=2-5]
      \arrow["{{\rho_V}}", from=1-7, to=5-7]
      \arrow["{{U \otimes u}}", from=2-3, to=2-5]
      \arrow["{{u \otimes U}}", from=2-3, to=4-3]
      \arrow["{{\rho_U}}", from=2-5, to=4-5]
      \arrow["{{\lambda_U}}", from=4-3, to=4-5]
      \arrow["{{f^{-1}}}"{description}, from=4-5, to=5-7]
      \arrow["{{I \otimes f}}"{description}, from=5-1, to=4-3]
      \arrow["{{\lambda_V}}", from=5-1, to=5-7]
    \end{tikzcd}
  \]
  Since $\lambda_U \circ (u \otimes U)$ is invertible and both $f^{-1}$ and $f \otimes f$ are invertible, the morphism $$\lambda_V \circ (v \otimes V) = f^{-1} \circ \lambda_U (u \otimes U) \circ (f \otimes f)$$ is invertible. This makes $v$ a central idempotent in $\C$. Since $f$ is in $\C$, the central idempotents $u$ and $v$ are equivalent.
\end{proof}

\begin{proposition} \label{prop:unique-halfbr}
    Let $\C$ be a monoidal category, $U \in \ob(\C)$, $\sigma$ and $\tau$ be half-braidings for $U$, and $u$ be a central idempotent $(U, \sigma) \to I$  and $(U, \tau) \to I$. Then $\sigma = \tau$.
\end{proposition}

\begin{proof}
  We adapt the proof of \cite[Lemma 2.6]{sheaf-rep}. Set $u' := (\lambda_U \circ (u \otimes U))^{-1}$ and $v' := (\lambda_U \circ (v \otimes U))^{-1}$, and let $A \in \ob(\C)$. Then for any distinct $\beta, \gamma \in \{\sigma, \tau\}$, the following diagram commutes; the respective identities used are the fact that $u$ respects the half-braidings $\gamma$ and $\rho^{-1} \circ \lambda$, the coherence theorem for monoidal categories, the naturality of $\lambda$, the naturality of $\alpha$, and the naturality of $\beta$:
  \begin{equation} \label{hbeq:1}
    \begin{tikzcd}
      {(A \otimes I) \otimes U} &&&& {A \otimes U} && {U \otimes A} \\
      \\
      && {(I \otimes A) \otimes U} && {I \otimes (A \otimes U)} && {I \otimes (U \otimes A)} \\
      \\
      {(A \otimes U) \otimes U} && {(U \otimes A) \otimes U} && {U \otimes (A \otimes U)} && {U \otimes (U \otimes A)}
      \arrow["{\rho_A \otimes U}", from=1-1, to=1-5]
      \arrow["{{\beta_A}}"', from=1-7, to=1-5]
      \arrow["{{\lambda_{U \otimes A}^{-1}}}", from=1-7, to=3-7]
      \arrow["{{\lambda_A \otimes U}}", from=3-3, to=1-5]
      \arrow["{{\alpha_{I A U}}}", from=3-3, to=3-5]
      \arrow["{{\lambda_{A \otimes U}}}"', from=3-5, to=1-5]
      \arrow["{{I \otimes \beta_A}}"', from=3-7, to=3-5]
      \arrow["{{u^{-1} \otimes (U \otimes A)}}", from=3-7, to=5-7]
      \arrow["{(A \otimes u) \otimes U}"', from=5-1, to=1-1]
      \arrow["{{(u \otimes A) \otimes U}}"', from=5-3, to=3-3]
      \arrow["{\gamma_A \otimes U}"', from=5-3, to=5-1]
      \arrow["{{u \otimes (A \otimes U)}}"', from=5-5, to=3-5]
      \arrow["{{\alpha_{U A U}^{-1}}}"', from=5-5, to=5-3]
      \arrow["{{U \otimes \beta_A}}"', from=5-7, to=5-5]
    \end{tikzcd}
  \end{equation}
  Next, observe that the following diagram commutes; the respective identities used are the fact that $\sigma_U = \id_{U \otimes U}$ since $U$ is a central idempotent \cite[Lemma 2.5]{sheaf-rep}, the half-braiding hexagon identity for $\sigma$, the naturality of $\tau$, the half-braiding hexagon identity for $\sigma$, and the fact that $\sigma_U = \id_{U \otimes U}$:
  \begin{equation} \label{hbeq:2}
    \begin{tikzcd}
      {U \otimes (U \otimes A)} &&&& {U \otimes (A \otimes U)} \\
      && {(U \otimes U) \otimes A} \\
      \\
      && {(U \otimes U) \otimes A} \\
      {U \otimes (U \otimes A)} &&&& {(U \otimes A) \otimes U} \\
      \\
      {U \otimes (A \otimes U)} &&&& {(A \otimes U) \otimes U} \\
      && {A \otimes (U \otimes U)} \\
      \\
      && {A \otimes (U \otimes U)} \\
      {(U \otimes A) \otimes U} &&&& {(A \otimes U) \otimes U}
      \arrow["{{U \otimes \sigma_A}}", from=1-1, to=1-5]
      \arrow[equals, from=1-1, to=5-1]
      \arrow["{{\alpha_{U A U}^{-1}}}", from=1-5, to=5-5]
      \arrow["{\alpha_{U U A}}", from=2-3, to=1-1]
      \arrow["{{\sigma_U \otimes A}}"', from=4-3, to=2-3]
      \arrow["{\alpha_{U U A}^{-1}}", from=5-1, to=4-3]
      \arrow["{{\sigma_{U \otimes A}}}"{description}, from=5-1, to=5-5]
      \arrow["{{U \otimes \tau_A}}", from=5-1, to=7-1]
      \arrow["{{\tau_A \otimes U}}", from=5-5, to=7-5]
      \arrow["{{\sigma_{A \otimes U}}}", from=7-1, to=7-5]
      \arrow["{{\alpha_{U A U}^{-1}}}", from=7-1, to=11-1]
      \arrow[equals, from=7-5, to=11-5]
      \arrow["{\alpha_{A U U}^{-1}}"', from=8-3, to=7-5]
      \arrow["{{A \otimes \sigma_U}}"', from=10-3, to=8-3]
      \arrow["{\sigma_A \otimes U}"{description}, from=11-1, to=11-5]
      \arrow["{\alpha_{A U U}}"', from=11-5, to=10-3]
    \end{tikzcd}
  \end{equation}
  The commutativity of diagram \ref{hbeq:1} means that the outermost pentagons of the following diagram commute, and the commutativity of diagram \ref{hbeq:2} means that the two morphisms in the center are equal:
  \[
    \begin{tikzcd}
      {U \otimes A} &&&&&& {A \otimes U} \\
      \\
      & {U \otimes (U \otimes A)} &&&& {(A \otimes U) \otimes U} \\
      \\
      {A \otimes U} &&&&&& {A \otimes U}
      \arrow["{(u^{-1} \otimes (U \otimes A)) \circ \lambda^{-1}_{U \otimes A}}"{pos=0.6}, from=1-1, to=3-2]
      \arrow["{\tau_A}"{description}, from=1-1, to=5-1]
      \arrow[equals, nfold, from=1-7, to=1-1]
      \arrow["{\sigma_A}", from=1-7, to=5-7]
      \arrow["{(\tau_A \otimes U) \circ \alpha_{U A U}^{-1} \circ (U \otimes \sigma_A)}", shift left=2, from=3-2, to=3-6]
      \arrow["{(\sigma_A \otimes U) \circ \alpha_{U A U}^{-1} \circ (U \otimes \tau_A)}"', shift right=2, from=3-2, to=3-6]
      \arrow["{(\rho_A \otimes U) \circ ((A \otimes u) \otimes U)}"'{pos=0.4}, from=3-6, to=5-7]
      \arrow[equals, nfold, from=5-1, to=5-7]
    \end{tikzcd}
  \]
  Thus, $\sigma = \tau$.
\end{proof}

The most important case of the above proposition is when $\C$ is a braided monoidal category. The braiding induces a half-braiding on each object, and these induced half-braidings are respected by any morphism. Thus, any central idempotent $u: (U, \sigma) \to I$ is also a central idempotent under the induced half-braiding, so $\sigma$ must coincide with this half-braiding.

Before saying anything further about the interaction between central idempotents and monoidal functors, we give a name to the following notion, introduced in the statement of \cite[Lemma 3.4]{sheaf-rep}, which gives a condition under which the image of a central idempotent under a lax monoidal functor is again a central idempotent in the target category.

\begin{definition}
  Let $\C$ and $\D$ be monoidal categories and $F: \C \to \D$ be a lax monoidal functor. We say $F$ \textbf{weakly preserves central idempotents} if $F$ preserves half-braidings of central idempotents, and if both $\eta^F$ and $\mu^F_{A U}$ are invertible for all $A \in \ob(\C)$ and $U \in \zi(\C)$.
\end{definition}

If $U$ is a central idempotent in $\C$ and $F$ weakly preserves central idempotents, the induced half-braiding on $F(U)$ is unique by Proposition \ref{prop:unique-halfbr}. Thus $\zi(-)$ can be viewed as a functor from the category of monoidal categories and lax monoidal functors which weakly preserve central idempotents into the category of sets. Next, we observe that this functor behaves well with respect to natural isomorphisms, so long as the target category is sufficiently well-behaved.

\begin{theorem}
  Let $\C$ and $\D$ be monoidal categories, let $F, G: \C \to \D$ be lax monoidal functors that weakly preserve central idempotents, and let $\varepsilon: F \to G$ be a natural isomorphism. We have $\zi(F) = \zi(G)$.
\end{theorem}

\begin{proof}
  Let $u: U \to I_\C$ be a central idempotent. Then $\zi(F)(u)$ and $\zi(G)(u)$ are central idempotents, and moreover, $\zi(G)(u) \circ \varepsilon_U: F(U) \to I$ is the same as the central idempotent $\zi(G)(u)$ up to an automorphism of $I$. Indeed, for any central idempotent $w \colon W \to I$, define $f_A = \rho_W \circ (W \otimes a) \circ \rho_W^{-1}$. Then
  \begin{align*}
    w \circ \rho_W \circ (W \otimes a) 
    & = \rho_I \circ (w \otimes a) \\
    & = \rho_I \circ (I \otimes a) \circ (w \otimes I) \\
    & = a \circ \rho_I \circ (w \otimes I) \\
    & = a \circ w \circ \rho_W,   
  \end{align*}
  so $w \circ f_a = a \circ w$, and $f_a$ is invertible. Applying Propositions~\ref{prop:equivalent-central-idempotents} and~\ref{prop:unique-halfbr}, we see that $\Z(F)(u) = \Z(G)(u)$ in $\Z(\D)$.
\end{proof}

This allows us to conclude that sufficiently well-behaved monoidal equivalences yield bijections of the central idempotent semilattices. In particular, we have:

\begin{corollary}
  Let $\C$ and $\D$ be monoidal categories, and let $F: \C \to \D$ be an equivalence of categories such that $F$ and $F^{-1}$ are lax monoidal functors that weakly preserve central idempotents. Then $\zi(F^{-1}) = \zi(F)^{-1}$.
\end{corollary}

It is especially useful to consider functors which not only preserve the central idempotents themselves, but also their finite meets.

\begin{definition} \label{def:stronglypreserves}
  (Adapted from \cite{sheaf-rep}). Let $\C$ and $\D$ be monoidal categories and $F: \C \to \D$ be a lax monoidal functor. We say $F$ \textbf{strongly preserves central idempotents} if $F$ weakly preserves central idempotents and the following diagrams commute for any $A \in \ob(\D)$ and $U, V \in \zi(\C)$:
  \[
    \begin{tikzcd}
      {F(I) \otimes A} && {I \otimes A} && \\
      &&&& A \\
      {A \otimes F(I)} && {A \otimes I}
      \arrow["{{(\eta^F)^{-1} \otimes A}}", from=1-1, to=1-3]
      \arrow["{{F(\sigma_I)_A}}"', from=1-1, to=3-1]
      \arrow["{{\lambda_A}}", from=1-3, to=2-5]
      \arrow["{{\rho^{-1}_A}}", from=2-5, to=3-3]
      \arrow["{{A \otimes \eta^F}}"', from=3-3, to=3-1]
    \end{tikzcd}
  \]
  \[
    \begin{tikzcd}
      {F(U \otimes V) \otimes A} && {A \otimes F(U \otimes V)} \\
      \\
      {(F(U) \otimes F(V)) \otimes A} && { A \otimes (F(U) \otimes F(V))} \\
      \\
      {F(U) \otimes (F(V) \otimes A)} && {(A \otimes F(U)) \otimes F(V)} \\
      \\
      {F(U) \otimes (A \otimes F(V))} && {(F(U) \otimes A) \otimes F(V)}
      \arrow["{{F(\sigma_{U \otimes V})_A}}", from=1-1, to=1-3]
      \arrow["{{\mu_{U V}^{-1} \otimes A}}", from=1-1, to=3-1]
      \arrow["{{\alpha_{F(U) F(V) A}}}", from=3-1, to=5-1]
      \arrow["{{A \otimes \mu_{U V}}}"', from=3-3, to=1-3]
      \arrow["{{F(U) \otimes F(\sigma_V)_A}}", from=5-1, to=7-1]
      \arrow["{{\alpha_{A F(U) F(V)}}}"', from=5-3, to=3-3]
      \arrow["{{\alpha^{-1}_{F(U) A F(V)}}}", from=7-1, to=7-3]
      \arrow["{{F(\sigma_U)_A \otimes F(V)}}"', from=7-3, to=5-3]
    \end{tikzcd}
  \]
\end{definition}

Notice that these conditions only regard the choice of induced half-braiding in the functor's target category, so functors into categories with suitably well-behaved half-braidings will satisfy them automatically. In particular, this is true of braided monoidal categories.

\begin{theorem} \label{thm:weakly-preserves-zi-implies-strongly-preserves}
  Let $\C$ and $\D$ be braided monoidal categories and $F: \C \to \D$ be a lax braided monoidal functor that weakly preserves central idempotents. Then $F$ strongly preserves central idempotents.
\end{theorem}

\begin{proof}
  For any $A \in \ob(\D)$, the square in the following diagram commutes by the naturality of $\sigma$, and the triangle commutes because $\sigma$ is a braiding, so the diagram as a whole commutes.
  \[
    \begin{tikzcd}
      {F(I) \otimes A} && {I \otimes A} && \\
      &&&& A \\
      {A \otimes F(I)} && {A \otimes I}
      \arrow["{{(\eta^F)^{-1} \otimes A}}", from=1-1, to=1-3]
      \arrow["{{\sigma_{F(I) A}}}"', from=1-1, to=3-1]
      \arrow["{{\lambda_A}}", from=1-3, to=2-5]
      \arrow["{{\sigma_{I A}}}", from=1-3, to=3-3]
      \arrow["{{\rho^{-1}_A}}", from=2-5, to=3-3]
      \arrow["{{A \otimes \eta^F}}"', from=3-3, to=3-1]
    \end{tikzcd}
  \]
  Additionally, for any $U, V \in \zi(\C)$, the square in the following diagram commutes by the naturality of $\sigma$, and the triangle commutes by the hexagon identities, so the diagram as a whole commutes.
  \[
    \begin{tikzcd}
      {F(U \otimes V) \otimes A} && {A \otimes F(U \otimes V)} \\
      \\
      {(F(U) \otimes F(V)) \otimes A} && { A \otimes (F(U) \otimes F(V))} \\
      \\
      {F(U) \otimes (F(V) \otimes A)} && {(A \otimes F(U)) \otimes F(V)} \\
      \\
      {F(U) \otimes (A \otimes F(V))} && {(F(U) \otimes A) \otimes F(V)}
      \arrow["{{\sigma_{F(U \otimes V) A}}}", from=1-1, to=1-3]
      \arrow["{{\mu_{U V}^{-1} \otimes A}}", from=1-1, to=3-1]
      \arrow["{{\sigma_{(F(U) \otimes F(V)) A}}}", from=3-1, to=3-3]
      \arrow["{{\alpha_{F(U) F(V) A}}}", from=3-1, to=5-1]
      \arrow["{{A \otimes \mu_{U V}}}"', from=3-3, to=1-3]
      \arrow["{{F(U) \otimes \sigma_{F(V) A}}}", from=5-1, to=7-1]
      \arrow["{{\alpha_{A F(U) F(V)}}}"', from=5-3, to=3-3]
      \arrow["{{\alpha^{-1}_{F(U) A F(V)}}}"', from=7-1, to=7-3]
      \arrow["{{\sigma_{F(U) A} \otimes F(V)}}"', from=7-3, to=5-3]
    \end{tikzcd}
  \]
  Thus, $F$ strongly preserves central idempotents.
\end{proof}

Observe that any strong monoidal functor automatically satisfies the invertibility conditions required to weakly preserve central idempotents, so for a given strong monoidal functor, we need only require that it preserves half-braidings in order to ensure that it weakly preserves central idempotents. Combining this observation with the well-behavedness of braided monoidal categories yields the following corollary.

\begin{corollary}
  Let $\C$ and $\D$ be braided monoidal categories and $F: \C \to \D$ be a strong braided monoidal functor. Then $F$ strongly preserves central idempotents.
\end{corollary}

We state as a separate corollary the special case where $F$ is the canonical functor from $\C$ to its strictification, as described, for example, in \cite{quasi-hopf}.

\begin{corollary} \label{cor:zi-strictification}
  Let $\C$ be a braided monoidal category and $\overline\C$ its strictification. Then $\zi(\C) \cong \zi(\overline\C)$ as meet-semilattices.
\end{corollary}

\bibliographystyle{amsplain}
\bibliography{refs}

\end{spacing}

\end{document}